\documentclass[a4paper]{article}

\usepackage[english]{babel}
\usepackage[utf8]{inputenc} 
\usepackage[T1]{fontenc}
\usepackage{amsthm}

\usepackage[a4paper,top=3cm,bottom=2cm,left=3cm,right=3cm,marginparwidth=1.75cm]{geometry}

\usepackage{amsmath}
\usepackage{amssymb}
\usepackage{float}
\usepackage{mathtools}
\usepackage{graphicx}
\usepackage{url}
\usepackage[colorinlistoftodos]{todonotes}
\usepackage[colorlinks=true, allcolors=blue]{hyperref}

\newtheorem{thm}{Theorem}[section]

\newtheorem{lem}[thm]{Lemma}

\newtheorem{defn}[thm]{Definition}

\numberwithin{equation}{subsection}

\title{\textbf{Existence and Uniqueness of Solutions to Nonlinear Diffusion with Memory}}
\author{\textbf{Yixian Chen}}

\begin{document}

\maketitle

\begin{abstract}
    This paper studies a nonlinear diffusion equation with memory:
    \[u_t=\nabla\cdot \big( D(x)\cdot\int_0^t K(t-s) \nabla\cdot\Phi(u(x,s))ds \big)+f(x,t).\]
    Where $K$ is memory Kernel and $D(x)$ is bounded. Under monotonicity and growth conditions on $\Phi$, the existence and uniqueness of weak solution is established. The analysis employs Orthogonal approximation, energy estimates, and monotone operator theory. The convolution structure is handled within variational frameworks. The result provides a basis for studying memory-type diffusion.
\end{abstract}

\section{Introduction}
Diffusion models with convolution-type memory were first introduced by \cite{Gurtin1968}, who incorporated a memory kernel $K(t)$ into the heat conduction equation to explain the finite propagation speed that the classical Fourier model fails to capture. Later, \cite{DAFERMOS1970554} systematically investigated such memory-type equations using energy methods and dissipation theory, and pointed out inherent difficulties concerning controllability. In parallel, \cite{JanP1993} developed a general framework combining Volterra integral equations with abstract semigroup theory, which laid the foundation for existence and uniqueness analysis of nonlinear memory PDEs.

\vspace{2mm}

\noindent In the early 2000s, mathematicians like \cite{Alesk2009} established the weak solution theory of Lions and Magenes to diffusion with memory. Later \cite{Zacher2015} combined fractional calculus with monotone operator methods to establish existence and uniqueness results for fractional and memory-driven diffusion. \cite{Li2012} further proved regularity results under degenerate kernels. \cite{SAKAMOTO2011426} contributed existence, uniqueness, and stability analyses for fractional diffusion models.

\vspace{2mm}

\noindent In the past five years, several authors advanced the theory of nonlinear diffusion with memory. \cite{Yuri2020} applied Carleman estimates and fractional integration techniques to prove weak solution uniqueness and address identifiability in related inverse problems. \cite{Alikhanov2017} construct a difference analog of the Caputo fractional derivative with generalized memory kernel (L1 formula). 

\vspace{2mm}

\noindent Despite these advances, research activity in this area has declined in recent years. On the one hand, Dafermos’s early “uncontrollability” result indicated that memory-type diffusion equations face fundamental limitations in exact control and observability. On the other hand, current PDE research trends have shifted toward fractional diffusion nonlocal operators and stochastic media models which are more attractive from both modeling and numerical perspectives. As a result, traditional convolution-type memory models, though still mathematically rich, have gradually become marginalized within the mainstream PDE community.

\pagebreak

\section{Main result}

\noindent We study a class of nonlinear diffusion equations with memory effects, given by the initial-boundary value problem:

\[
\left\{
\begin{aligned}
    &u_t = \nabla \cdot \left( D(x) \int_0^t K(t-s) \nabla \Phi(u(x,s)) \, ds \right) + f(x,t), && x \in \Omega, \ t \in [0,T] \\
    &u(x,0) = u_0(x), && x \in \Omega, \ t = 0, \\
    &u(x,t) = 0, && x \in \partial \Omega, \ t \in [0,T].
\end{aligned}
\right.
\]

\noindent Where $\Omega\subset \mathbb R^d$ is an open bounded Lipschitz domain, $D(x)$ is bounded uniformly elliptic coefficient, and $K\in L^1(0,T)$ is a memory kernel. The unknown $u(x,t)$ represents a diffusive state influenced by both spatial heterogeneity and temporal memory.

\begin{thm}
    Let $\Omega\subseteq \mathbb R^d$ be open bounded with Lipschitz boundary, $T>0$ be fixed, and memory kernel $0<K=K(t)\in L^1(0,T)$, nonlinear diffusion $\Phi\in C^1(\mathbb R)$ whose derivative is bounded by growth condition:
    \[\forall \xi\in\mathbb R,\ |\Phi'(\xi)|\leq L(1+|\xi|^{m-1}).\]
    With constant that $L>0$ and $2 \leq m\leq \min\{3,\frac{2d}{d-2}\}$, and assume the monotonicity condition: $\Phi'(\xi)\leq 0$. Let $D(x)$ be diffusion coefficient uniformly bounded by $D_{\min}$ and $D_{\max}$. Let external force term $f\in L^2(0,T;H^{-1}(\Omega))$, and $u_0\in L^2(\Omega)$ be the initial condition. Then $\exists u\in L^2(0,T;H_0^1(\Omega))$ such that for all test functions: $\varphi\in C_c^\infty(\Omega\times(0,T))$ is the solution of the partial differential equaiton, moreover the solution is unique in the sense of distributive.
\end{thm}

\section{Existence and Uniqueness of Solution in $V_N$ Space}

\begin{defn}
    Let $\{\vec{e}_i\}$ be orthogonal basis of $H_0^1(\Omega)$ consisting of eigenfunctions of the negative Dirichlet Laplacian:
    \[-\Delta \vec{e}_i=\lambda_i \vec{e}_i\text{ in }\Omega,\ \text{and }\vec{e}_i=0\text{ on }\partial\Omega.\]
    Then define the finite-dimensional:
    \[V_N:=span\{\vec{e}_1,\dots,\vec{e}_N\}\subseteq H_0^1(\Omega).\]
\end{defn}
\begin{defn}
    The projection $P^N$ is defined as:
    \[P^N:H_0^1(\Omega)\to V_N.\]
    Thus $u^N$ is defined as:
    \[u^N:=P^N u=\sum_1^N \big(\int_\Omega u\cdot\vec{e}_i \big)\vec{e}_i=\sum_1^N c_i(t)\vec{e}_i(x).\]
\end{defn}

\begin{lem}
    \cite{Perko2001}
    Let $F=F(u)$ and $F:\mathbb R^N\times[0,T]\to \mathbb R^N$ be continuous and locally Lipschitz in $u=u(t)$. Suppose $u(t)$ solves $u'=F(u,t)$ on $[0,T^*)$ and $\parallel u\parallel\leq M$, then $u(t)$ can be extended to $[0,T]$ as a classical solution. But only works in finite dimensional spaces.
    \label{perko}
\end{lem}

\begin{lem}
    \cite{beckner1975} Assume $f\in L^p(\mathbb R^d)$ and $L^q(\mathbb R^d)$ with:
    \[\frac{1}{p}+\frac{1}{q}=\frac{1}{r}+1; \ p,q,r\in [1,+\infty].\]
    There exists:
    \[\parallel f*g\parallel_{L^r}\leq \parallel f\parallel_{L^p}\cdot\parallel g\parallel_{L^q}.\]
    \label{beckner}
\end{lem}

\begin{thm}
    Let $V_N:=span\{ \vec{e}_1,\dots, \vec{e}_N \}$ be a finite-dimensional subspace. $\exists !u^N\in L^2(0,T;V_n)\subseteq L^2(0,T;H_0^1(\Omega))$ satisfy the weak formula:
    \[\int_0^T \int_\Omega \partial_t u^N\varphi+\int_0^T\int_\Omega D(x)\bigg(M(u^N(x,t))\bigg) \cdot\nabla\varphi=\int_0^T\int_\Omega f\varphi.\]
\end{thm}
\begin{proof}
    Define the middle term:
    \[a(t,u^N,\varphi):=\int_\Omega D(x)\bigg( \int_0^t K(t-s)\nabla \Phi(u^N) ds \bigg)\nabla \varphi dx.\]
    Using test function $\varphi:=\vec{e}_i$:
    \[c_i '(t)+a(t,u^N,\vec{e}_i)=\int_\Omega f\cdot\vec{e}_i.\]
    Then define system:
    \[\begin{pmatrix}
        c_1'(t)\\c_2'(t)\\\dots\\c_N'(t)
    \end{pmatrix}=:c'(t)=F(t,c(t)):=(F_j(t,c(t)))_{N\times 1}:=\begin{pmatrix}
        \big<f,\vec{e}_1\big>_{L^2(\Omega)}-a(t,u^N,\vec{e_1})\\
        \big<f,\vec{e}_2\big>_{L^2(\Omega)}-a(t,u^N,\vec{e_2})\\
        \dots\\
        \big<f,\vec{e}_N\big>_{L^2(\Omega)}-a(t,u^N,\vec{e_N})
    \end{pmatrix}\]
    Which is an ODE system. Continuous dependence on initial condition yielding a perturbation from $c$ to $\tilde{c}$. Then: 
        \[|\nabla \Phi(u^N)-\nabla \Phi(\tilde{u}^N)|\leq L(1+|u^N|^{m-2}+|u|^{m-2})|u^N-\tilde{u}^N|.\]
        It's well posed because of the range of $m\in [2,\min\{3,\frac{2d}{d-2}\}]$. While,
        \[|u^N-\tilde{u}^N|=|\sum_1^N  (c_i-\tilde{c}_i)\vec{e}_i|\leq \sqrt{\frac{\sum_1^N (c_i-\tilde{c}_i)^2 }{N}}\cdot\sqrt{|\sum_1^N \big< \vec{e}_i,\vec{e}_i \big> |} =  \sqrt{\frac{\sum_1^N (c_i-\tilde{c}_i)^2 }{N}} \cdot \sqrt{N} =\parallel c-\tilde{c}\parallel_{\mathbb R^N}.\]
        Then for integral of nonlinear term, and again apply \eqref{beckner} to reduce the convolution part of $K$, there is:
        \[\parallel F(c,t)- F(\tilde{c},t) \parallel_{\mathbb R^N}\leq\parallel D\parallel_{L^\infty}\cdot\parallel K\parallel_{L^1 }\cdot\parallel c-\tilde{c} \parallel_{\mathbb R^N}\int_0^t\int_\Omega L(1+|u^N|^{m-2}+|\tilde{u}^N|^{m-2}).\]
    Since $\forall t\in [0,T]$, and by Minkovskii inequality:
    \[\int_0^t\int_\Omega L(1+|u^N|^{m-2}+|\tilde{u}^N|^{m-2})\leq L(T|\Omega|+\parallel u^N\parallel_{L^2(0,T;L^2(\Omega))}^{\frac{m-2}{2}}+\parallel\tilde{u}^N\parallel_{L^2(0,T;L^2(\Omega))}^{\frac{m-2}{2}}).\]
    Which is bounded in $V_N$. Thus, we know that $F(c,t)$ is locally Lipschitz. Then by Picard-Lindelof theorem, a unique local solution exists on: $u^N\in L^2(0,T^*;V_N)$ where $T^*\leq T$ is a constant. Our interest is whether a global solution can be uniquely extended. Using \eqref{perko}, we have the unique extended solution. 
\end{proof}

\section{Energy Estimate and Convergence}
In this chapter, we construct approximate solutions. Our goal is to rigorously derive the uniform energy estimates which will be used further to prove the existence of weak solutions.

\begin{lem}
    \cite{bainov1992}
    Let $u$ $a$ $k$ be non-negative continuous functions in interval $[\alpha,\beta]$ and suppose:
    \[u(t)\leq a(t)+\int_\alpha^t k(s)u^p(s).\]
    Where $0<p<1$, then 
    \[u(t)\leq a(t)+x_0^p\bigg(\int_\alpha^t k^{\frac{1}{q}}(s)ds  \bigg)^{q}.\]
    Where $q=1-p$, $x_0$ be the unique positive root:
    \[x-c_1-c_2x^p=0,\ c_1=\int_\alpha^\beta \alpha(s)ds,\ c_2=\int_\alpha^\beta\bigg(\int_\alpha^t k^{\frac{1}{q}}(s)ds\bigg)^qdt.\]
    \label{bainov}
\end{lem}

\begin{lem}
    \cite{friedman1980} For $\sigma$ finite measure spaces $X,Y$, $f$ is a measurable function:
    \[\int_X\int_Yf(x,y)=\int_Y\int_Xf(x,y)=\int_{X\times Y}f(x,y).\]
    If one of the above integral is finite.
    \label{friedman}
\end{lem}

\begin{thm}
    $\parallel u^N\parallel_{L^2(0,T;H_0^1(\Omega))}$, $\parallel \partial_t u^N\parallel_{L^2(0,T;H^{-1}(\Omega))}$ are bounded by a constant independent with $N$.
\end{thm}
\begin{proof}
    Take test function $\varphi:=u^N$, then:
    \[\frac{1}{2}\frac{d}{dt}\parallel u^N(t)\parallel_{L^2(\Omega)}^2=\int_\Omega \partial_t u^N\cdot u^N.\]
    Define this term as the energy $E(t)$, then:
    \[\frac{1}{2}\frac{d}{dt} E(t)=\int_\Omega f\cdot u^N-D\cdot\big(\int_0^t K(t-s)\Phi'(u^N)\nabla u^N(x,s)\big)\cdot\nabla u^N.\]
    The key point is to estimate the middle term:
    \[\bigg|\int_\Omega D\cdot\int_0^t K(t-s)\Phi'(u^N(x,s))\cdot\nabla u^N(x,t)\bigg|.\]
    First use \eqref{friedman} and unifom bound of $D$, we have:
    \[\parallel D\parallel_{L^\infty}\cdot \int_0^t \int_\Omega \bigg|K(t-s)\Phi'(u^N(x,s))\cdot\nabla u^N(x,t)\bigg|.\]
    Using Cauchy inequality and young inequality:
    \[\parallel D\parallel_{L^\infty} \int_0^t K(t-s)\bigg(\frac{1}{2\lambda}\parallel \Phi'(u^N(s))\parallel_{L^2(\Omega)} +\frac{\lambda}{2}\parallel \nabla u^N(t)\parallel_{L^2(\Omega)} \bigg).\]
    Now the integral is split into 2 parts, the first part is easy,
    \[\parallel D\parallel_{L^\infty}\int_0^t K(t-s)\frac{\lambda}{2}\parallel\nabla u^N(t)\parallel_{L^2(\Omega)}\leq  \parallel D\parallel_{L^\infty}\cdot \parallel K\parallel_{L^1}\cdot\frac{\lambda}{2}\parallel \nabla u^N(t)\parallel_{L^2(\Omega)}.\]
    The second part need growth condition:
    \[\leq \parallel D\parallel_{L^\infty}\frac{1}{2\lambda}\int_0^t K(t-s)\parallel L(1+|u^N(s)|^{m-1})\parallel_{L^2(\Omega)}.\]
    According to Minkovskii inequality and apply \eqref{beckner}:
    \[\leq \parallel D\parallel_{L^\infty}\frac{1}{2\lambda}\parallel K\parallel_{L^1}\cdot L\bigg( |\Omega|^{\frac{1}{2}}+ |\Omega|^{\frac{3-m}{2}} \int_0^t \parallel u^N(s)\parallel_{L^2(\Omega)}^{\frac{m-1}{2}} \bigg).\]
    Since $\Phi'\leq 0$, so there is inequality:
    \[\frac{1}{2}|\frac{d}{dt}E(t)|\leq \bigg|\parallel D\parallel_{L^\infty} \parallel K\parallel_{L^1}\bigg[ \frac{\lambda}{2}\parallel \nabla u^N\parallel_{L^2(\Omega)}  .\]
    \[+\frac{1}{2\lambda} \cdot L \cdot\bigg( |\Omega|^{\frac{1}{2}} +|\Omega|^{\frac{3-m}{2}} \int_0^t E^{\frac{m-1}{4}}(s) \bigg)\bigg]- \parallel f\parallel_{H^{-1}(\Omega)}\cdot\parallel u^N\parallel_{L^2(\Omega)}\bigg|.\]
    Since the minus sign lies within the absolute value, we focus on estimating the following term:
    \[\parallel D\parallel_{L^\infty}\cdot\parallel K\parallel_{L^1}\cdot\frac{\lambda}{2}\cdot\parallel\nabla u^N\parallel_{L^2(\Omega)}-\parallel f\parallel_{H^{-1}}\cdot\parallel u^N\parallel_{L^2(\Omega)}.\]
    By Poincare inequality from zero trace:
    \[\parallel u^N\parallel_{L^2(\Omega)}\leq C_\Omega\cdot\parallel \nabla u^N\parallel_{L^2(\Omega)}.\]
    Therefore, by selecting:
    \[\frac{\lambda}{2}= \frac{C_\Omega \cdot\parallel f\parallel_{H^{-1}(\Omega)}}{\parallel D\parallel_{L^\infty}\cdot\parallel K\parallel_{L^1}}.\]
    We know the value in absolute sign remain positive, thus:
    \[\frac{1}{2}|\frac{d}{dt}E(t)|\leq\parallel D\parallel_{L^\infty} \parallel K\parallel_{L^1}\bigg[ \frac{\lambda}{2}\parallel \nabla u^N\parallel_{L^2(\Omega)}.\]
    \[+\frac{1}{2\lambda} \cdot L \cdot\bigg( |\Omega|^{\frac{1}{2}} +|\Omega|^{\frac{3-m}{2}} \int_0^t E^{\frac{m-1}{4}}(s) \bigg)\bigg]- \parallel f\parallel_{H^{-1}(\Omega)}\cdot\parallel u^N\parallel_{L^2(\Omega)}.\]
    \[\leq \parallel D\parallel_{L^\infty} \parallel K\parallel_{L^1}\cdot\frac{1}{2\lambda} \cdot L \cdot\bigg( |\Omega|^{\frac{1}{2}} +|\Omega|^{\frac{3-m}{2}} \int_0^t E^{\frac{m-1}{4}}(s) \bigg).\]
    Which is a Gronwall type inequality, if we denote:
    \[C_1:=\frac{\parallel D\parallel_{L^\infty}^2\cdot \parallel K\parallel_{L^1}^2\cdot L\cdot |\Omega|^{\frac{1}{2}}}{2C_\Omega\cdot\parallel f\parallel_{H^{-1}(\Omega)}}.\]
    \[C_2:=\frac{\parallel D\parallel_{L^\infty}^2\cdot \parallel K\parallel_{L^1}^2\cdot L\cdot |\Omega|^{\frac{3-m}{2}}}{2C_\Omega\cdot\parallel f\parallel_{H^{-1}(\Omega)}}.\]
    \[p:=\frac{m-1}{4}.\]
    That is:
    \[|\frac{d}{dt}E(t)|\leq C_1+C_2\int_0^t E^{p}(s)ds.\]
    Since $p=\frac{m-1}{4}\leq \frac{1}{2}<1$, then:
    \[E(t)-E_0=\int_0^t E'(s)\leq \int_0^t C_1+C_2\int_0^s E^p(u).\]
    Thus a new inequality is deduced:
    \[E(t)\leq E(0)+C_1 t+C_2\int_0^t \int_0^s E^p(r)drds.\]
    Then we can apply the nonlinear estimate inequality \eqref{bainov} and get:
    \[E(t)\leq E(0)+C_1 t+x_0^p C_2 t^{1-p}.\]
    Where: $x_0$ is the only root of:
    \[ x=E_0T+\frac{1}{2}C_1T^2+\frac{C_2T^{2-p}}{1-p}x^p.\]
    Then we have the estimate on $E(t)$, in other words: $\parallel u^N\parallel_{L^2(0,T;L^2(\Omega))}^2$:
    \[x_0\leq E_0T+\frac{1}{2}C_1T^2+\frac{C_2T^{2-p}}{1-p}\cdot\max\{1,T^p\}.\]
    \[0\leq E(t)\leq E_0+C_1T+x_0^p\cdot C_2\cdot \max\{1,T,T^{1-p}\}.\]
    Observe that the use of the $\max{}$ function addresses the case $T<1$ which would otherwise lead to the inequality $T^p>T$ when $p<1$. As a result, $E(t)$ is uniformly bounded by a constant, denoted by $J$. Returning to the original equation, we obtain:
    \[\int_0^T \int_\Omega \partial_t u^N\cdot u^N\leq J.\]
    Recall that:
    \[|\frac{d}{dt}E(t)|\leq C_1+C_2\int_0^t E^{p}(s)ds.\]
    Is a bound for $E'(t)$, by uniform bound of $E$ on right side:
    \[|E'(t)|\leq C_1+C_2+J^p T.\]
    We finalized our conclusion. Because, back to the equation:
    \[E'(t)=2\int_\Omega \partial_tu\cdot u.\]
    First we apply Cauchy inequality, this process is well-posed, because we know $E'$ is uniformly bounded. Then:
    \[\parallel \partial_t u(t)\parallel_{H^{-1}(\Omega)}^2+\parallel u(t)\parallel_{H_0^1(\Omega)}^2\leq C_1+C_2+J^pT.\]
    Then integral over time gives the upper uniform bound:
    \[\parallel \partial_t u\parallel_{L^2(0,T;H^{-1}(\Omega))}^2+\parallel u\parallel_{L^2(0,T;H_0^1(\Omega)}\leq 2\sqrt{T}(C_1+C_2+J^p T).\]
    One may observe that:
    \[(\int_0^T 1^2 )^{\frac{1}{2}}=\sqrt{T}.\]
    In wchich:
    \[C_1=\frac{\parallel D\parallel_{L^\infty}^2\cdot \parallel K\parallel_{L^1}^2\cdot L\cdot |\Omega|^{\frac{1}{2}}}{2C_\Omega\cdot\parallel f\parallel_{H^{-1}(\Omega)}}.\]
    \[C_2=\frac{\parallel D\parallel_{L^\infty}^2\cdot \parallel K\parallel_{L^1}^2\cdot L\cdot |\Omega|^{\frac{3-m}{2}}}{2C_\Omega\cdot\parallel f\parallel_{H^{-1}(\Omega)}}.\]
    \[J=\parallel u_0\parallel_{L^2(\Omega)}+C_1T+\bigg( \parallel u_0\parallel_{L^2(\Omega)}T+\frac{1}{2}C_1T^2+\frac{C_2T^{2-p}}{1-p}\cdot\max\{1,T^p\} \bigg)^p\cdot C_2\cdot\max\{ 1,T,T^{1-p} \}.\]
    \[p=\frac{m-1}{4}.\]
\end{proof}

\begin{lem}
    \cite{kothe1960}
    $X$ be normed linear space, $X^*$ be its dual space. The unit closed ball:
    \[B:=\{f\in X^*:\parallel f\parallel_{X^*}\leq 1\}\]
    is compact under weak* topology.
    \label{kothe}
\end{lem}

\begin{lem}
    Let $\{u^N\}$ be $L^2(0,T;H_0^1(\Omega))$ function sequences with energy estimates bounded by a constant independent with $N$, $\exists u\in L^2(0,T;H_0^1(\Omega))$, such that:
    \[u^N(t,\cdot)\rightharpoonup u(t,\cdot),\ \partial_t u^N(t,\cdot) \rightharpoonup v(t,\cdot)\]
    by weak convergence in $H_0^1(\Omega)$ and $H^{-1}(\Omega)$. i.e. the 2 sequences all have their own limitation.
\end{lem}

\begin{proof}
    Rellich-Kondrachov theorem implies:
    \[H_0^1(\Omega)\hookrightarrow\hookrightarrow L^2(\Omega)\]
    Thus an embedding chain exists:
    \[H_0^1\hookrightarrow\hookrightarrow L^2\hookrightarrow H^{-1}\]
    Hereby with energy estimate: $u^N(t,\cdot)$ bounded in $H_0^1(\Omega)$ and $\partial_t u^N$ bounded in $H^{-1}(\Omega)$, we immediately know by \eqref{kothe} that:
    \[\{u^{N}(t,\cdot)\}\rightharpoonup u(t,\cdot)\in H_0^1(\Omega)\]
    by weak convergence. And 
    \[\{\partial_t u^{N}(t,\cdot)\}\rightharpoonup v(t,\cdot)\in H^{-1}(\Omega)\]
    But we still have to classify whether the identity $\partial_t u=v$ holds?
\end{proof}

\begin{thm}
    For $\{u^N\}\subset L^2(0,T;H_0^1(\Omega))$ with condition that: $\partial_t u^N(t,\cdot) \rightharpoonup v(t,\cdot)\in H^{-1}(\Omega)$ and $u^N(t,\cdot)\rightharpoonup u(t,\cdot)\in H_0^1(\Omega)$, then
    \[\partial_t u=v\]
    in the sense of distributional derivative, i.e. for all test functions: $\phi\in C_c^\infty(0,T)$ and $\psi\in H_0^1(\Omega)$, 
    \[\big<\partial_t u,\phi\psi\big>_{H^{-1},H_0^1}=-\big< u,\partial_t \phi\psi\big>_{H^{-1},H_0^1}\]
\end{thm}

\begin{proof}
    Use weak convergence for $u^N(t,\cdot)$ and $\partial_t u^N(t,\cdot)$:
    \[\lim \int_0^T \big< u^N,\partial_t\phi(t)\psi\big>_{H^{-1},H_0^1}=\int_0^T \big< u,\partial_t\phi(t)\psi\big>_{H^{-1},H_0^1}\]
    (Above is well-posed because $\partial_t \phi\psi\in L^2(0,T;H_0^1(\Omega))$) and that:
    \[\lim \int_0^T \big< \partial_t u^N,\phi(t)\psi\big>_{H^{-1},H_0^1}=\int_0^T \big< u',\phi(t)\psi\big>_{H^{-1},H_0^1}\]
    For each $u^N$, its distributional derivative satisfy:
    \[-\big< u^N,\partial_t\phi(t)\psi\big>_{H^{-1},H_0^1}=\big< \partial_t u^N,\phi(t)\psi\big>_{H^{-1},H_0^1}\]
    By take $N\to\infty$, we have the desired formula with $\forall \phi\in C_c^\infty(0,T)$ and $\psi\in H_0^1(\Omega)$. Now use the fact that:
    \[\overline{C_c^\infty(0,T)\otimes H_0^1(\Omega)\cap L^2(0,T;H_0^1(\Omega))}=L^2(0,T;H_0^1(\Omega))\]
    and the fact that distributional derivative is uniquely determined by dual pairing.
\end{proof}

\begin{thm}
    Let $\{u^N\}$ be $L^2(0,T;H_0^1(\Omega))$ function sequences with energy estimates bounded by a constant independent with $N$, and $u^N(t,\cdot)$ weakly converge to $u(t,\cdot)\in H_0^1(\Omega)$, then without loss of generalcity:
    \[u^N\to u\]
    by strongly convergence in $L^2(0,T;L^2(\Omega))$.
\end{thm}

\begin{proof}
    The result is trivial, take the integral, use \eqref{friedman} change the order of integral, then use the uniform bound deduced above, the difference is less or equal to a multiplication of a bounded number and small number.
\end{proof}

\section{Identification of Limit as a Solution}
Last chapter provide us with some basic results:
\[
\left\{
\begin{aligned}
    & u^N(t,\cdot)\rightharpoonup u(t,\cdot), && H_0^1(\Omega)\\
    & u^N\to u, && L^2(0,T;L^2(\Omega))\\
    &\partial_t u^N(t,\cdot)\rightharpoonup \partial_t u(t,\cdot) && H^{-1}(\Omega)
\end{aligned}
\right.
\]
This section aims to verify that the limit function $u$ satisfies the weak formulation of the original nonlinear equation. That is, for any test function $\varphi\in C_c^\infty(\Omega\times(0,T))$,
\[-\int_0^T \int_\Omega u\partial_t\varphi=-\int_0^T\int_\Omega D(x)\bigg(K(t-s)\nabla \Phi(u(x,s))\bigg) \cdot\nabla\varphi+\int_0^T\int_\Omega f\varphi\]

\begin{lem}
    \cite{bogachev2007} $\Omega$ be bounded, let $\{u^{N}\}$ be a sequence of function with $u^N\to u$ by strong convergence in $L^2(0,T;L^2(\Omega))$ such that $\exists C>0$:
    \[\sup \parallel u^N \parallel_{L^2(0,T;L^2(\Omega))}\leq C \]
    Then $\exists$ subsequence $\{ N_k\}$ of $\{N\}$:
    \[u^{N_k}\to u\text{ a.e. }\Omega\times (0,T)\]
    \label{bogachev}
\end{lem}

\begin{lem}
    For $\{u^N\}\subseteq L^2(0,T;H_0^1(\Omega))$ with uniform norm bound, and $u^N \to u$ strongly in $L^2(0,T;L^2(\Omega))$, then:
    \begin{itemize}
        \item $u\in L^2(0,T;H_0^1(\Omega))$;
        \item $\nabla u^N \rightharpoonup \nabla u$ in $L^2(0,T;L^2(\Omega))$.
    \end{itemize}
    \label{Gradient weak convergence}
\end{lem}
\begin{proof}
    Let us assume $u\in H_0^1$. Observe that:
    \[H_0^1(\Omega)\hookrightarrow\hookrightarrow L^2(\Omega)\]
    And that $H_0^1(\Omega)$ is reflexive Banach space. Consequently, $L^2(0,T;H_0^1(\Omega))$ is also reflexive and weakly sequentially closed. Moreover, the uniform bound holds: 
    \[\sup_N \parallel u^N\parallel_{L^2(0,T;H_0^1(\Omega))}\leq c\]
    By \eqref{bogachev}, $\exists$ subsequence $\{N_k\}$,
    \[u^{N_k}\to u\text{ a.e. }\]
    Since $u^N\in L^2(0,T;H_0^1(\Omega))$, it follows that $\nabla u^N\in L^2(0,T;L^2(\Omega)$, by the compactness result \eqref{kothe}, there exists a subsequence $u^{N_m}$ such that:
    \[\nabla u^{N_m}\rightharpoonup v\]
    For some $v\in L^2(0,T;L^2(\Omega))$. Suppose that $\nabla u^{N_m}(t,\cdot)\rightharpoonup v(t,\cdot)$ weakly in $L^2(\Omega)$. Then for any test function $g\in L^2(\Omega)$, it follows that:
    \[\int_\Omega \nabla u^{N_m} g\to \int_\Omega v g\]
    Applying \eqref{bogachev} to the sequence $\{\nabla u^{N_m}\}$, we obtain:
    \[\int_\Omega \nabla u^{N_n} g\to \int_\Omega v g,\ \nabla u^{N_n}\to \nabla u\text{ a.e.}\]
    Choosing $(c+1)|g|$ as a dominating function, the Dominated Convergence Theorem yields:
    \[\int_{\Omega} \nabla u^{N_n}g\to \int_\Omega \nabla u g\]
    Consequently,
    \[\int_\Omega vg=\int_\Omega \nabla u g,\ \forall g\in L^2(\Omega)\]
    which implies that $\nabla u=v$. It remains to show that this convergence holds for the entire sequence, not just for a subsequence.

    We distinguish two points:
    \begin{itemize}
        \item Whether the sequence $\{\nabla u^N\}$ converges? 
        \item Whether its weak limit coincides with $\nabla u$?
    \end{itemize}
    Suppose $\nabla u^N$ does not converge. Since it is bounded in a reflexive Banach space, and the embedding is compact \eqref{kothe}, this would contradict the fact that every Cauchy sequence must converge in a compact space. For instance, one may normalize the sequence by setting $f:=\frac{\nabla u^N}{\parallel \nabla u^N\parallel}$, and apply \eqref{kothe} to deduce compactness—leading to a contradiction. The second problem can be settled by supposing that the weak limit exists but is not equal to $\nabla u$. This would contradict the almost everywhere convergence of a subsequence $\nabla u^{N_k}\to u$ and the uniqueness of weak limits in Hilbert spaces. Finally, to verify that $u\in H_0^1(\Omega)$, observe that there exists $h\in L^2(\Omega)$ such that $\nabla u^N\rightharpoonup h$, $u^N\to u$ strongly in $L^2(0,T;L^2(\Omega))$ and $u^{N_k}\to u$ almost everywhere. By the closedness of $L^2(\Omega)$ in $H_1(\Omega)$, it follows that $\nabla u=h$ and hence $u(t,\cdot)\in H_0^1(\Omega)$.
\end{proof}

\begin{lem}
    Let all assumptions remain as above. Then
    \[\nabla\Phi(u^N)(t,\cdot)\rightharpoonup \nabla \Phi(u)(t,\cdot)\]
    \label{Nonlinear term weak convergence}
\end{lem}
\begin{proof}
    By the chain rule:
    \[\nabla \Phi(u^N)=\Phi'(u^N)\nabla u^N\]
    By \eqref{bogachev}, there exists a subsequence $\{u^{N_k}\}$, such that $u^{N_k}\to u$ almost everywhere, since $\Phi$ in continuous, we have \( \Phi'(u^{N_k}) \to \Phi'(u) \) almost everywhere as well. Moreover, from \eqref{Gradient weak convergence}, \( \nabla u^N \rightharpoonup \nabla u \) in \( L^2 \). To show weak convergence of the nonlinear term, let \( \varphi \in C_c^\infty(\Omega) \). Then:
    \[\int_\Omega \Phi'(u^N)\nabla u^N\varphi-\int_\Omega \Phi'(u)\nabla u \varphi=\int_\Omega \Phi'(u^N)(\nabla u^N-\nabla u )\varphi+\int_\Omega (\Phi'(u^N)-\Phi'(u))\nabla u\varphi\]
    We handle both terms separately. For the first term:
    \[L(1+|u^N|^{m-1})\varphi\]
    we consider take abover formula as new test function, hence its $L^2$ integrability is guaranteed by range of $m\in [1,\min\{3,\frac{2d}{d-2}\}]$ and sobolev embedding. For the second term:
    \[\bigg|\int_\Omega (\Phi'(u^N)-\Phi'(u))\nabla u\varphi\bigg|\leq \parallel \Phi'(u^N)-\Phi'(u)\parallel_{L^\infty}\bigg|\int_\Omega \nabla u\varphi \bigg|\]
    Observe that the inner integral must be finite, since $u(t,\cdot)\in H_0^1(\Omega)$, 
    \[\bigg|\int_\Omega \nabla u\varphi\bigg|\leq \parallel\nabla u\parallel_{L^2}\cdot\parallel \varphi\parallel_{L^2}\]
    It remains to verify whether $\parallel \Phi'(u^N)-\Phi'(u)\parallel_{L^\infty}=0$. Otherwise, there would exist a measurable subset $E\subset \Omega\times(0,T)$, of positive measure such that $\Phi'(u^N)\neq \Phi'(u)$ on $E$, which would contradict the almost everywhere convergence of $u^N\to u$ and the continuity of $\Phi'$.
\end{proof}

\begin{lem}
    Let $X$, $Y$ be two Banach space, with a bounded linear operator: $T\in\mathcal{L}(X,Y)$, $\{x_n\}\in X$ be a sequence weakly converges to $x\in X$, then:
    \[Tx_n\rightharpoonup Tx\text{ in }Y\]
    \label{Weak convergence preserve by bounded linear functional}
\end{lem}

\begin{proof}
    For $x_n\rightharpoonup x$, meaning for all linear continuous functionl $f\in X^*$, 
    \[f(x)=\lim f(x_n)\]
    One may observe that a bounded linear functional in a Banach space is equivalent to continuous linear functional. Then by defining a composite based on $\forall g\in Y^*$ that $f:=g(T)$, it's clear that $f$ is linear continuous since $g$ is linear continuous functional and $T$ is linear bounded. That means:
    \[\lim g(T)x_n=g(T)x,\ \forall g\in Y^*\]
    Therefore: $Tx_n\rightharpoonup Tx\in Y$. 
\end{proof}

\begin{thm}
    With all conditions and assumptions as above, 
    \[K*\nabla \Phi(u^N(x,t))\rightharpoonup K*\nabla \Phi(u(x,t))\text{ in }L^2(0,T;L^2(\Omega;\mathbb R^d))\]
    \label{Memory kernel weak convergence}
\end{thm}
\begin{proof}
    From discussion in \eqref{Nonlinear term weak convergence}, there's result:
    \[\nabla \Phi(u^N)\in L^2(0,T;L^2(\Omega;\mathbb R^d))\]
    With assumption $K\in L^1(0,T)$, and introduce definition:
    \[g^N:=\nabla \Phi(u^N(\cdot,t))\in L^2(\Omega;\mathbb R^d)\]
    Applying \eqref{beckner} yields the following estimate:
    \[\parallel K*g^N\parallel_{L^2(0,T;L^2(\Omega;\mathbb R^d))}\leq \parallel K\parallel_{L^1(0,T)}\cdot\parallel g^N\parallel_{L^2(0,T;L^2(\Omega;\mathbb R^d))}\]
    This shows the regularity of $K*\nabla\Phi(u^N)$, finiteness meaning it's in the space $L^2(0,T;L^2(\Omega;\mathbb R^d))$. Now we check weak convergence. The technique is to show the convolution operator $\cdot \mapsto K*(\cdot)$ is a bounder linear functional on $L^2(0,T;L^2(\Omega;\mathbb R^d))$, then by \eqref{Weak convergence preserve by bounded linear functional}, we can conclude the weak convergence. Let: $L: (\cdot)\mapsto K*(\cdot)$, denote $L^2(\Omega;\mathbb R^d)=:X$, we have:
    \[\forall g\in X,\ Tg:=\int_0^tK(t-s)g(s)ds\]
    Estimate over $X$, we have:
    \[\parallel Tg(t)\parallel_X\leq \int_0^T |K(t-s)|\cdot\parallel g(s)\parallel_X ds\]
    Take $L^2$ norm over time and apply Minkovskii inequality:
    \[\parallel Tg\parallel_{L^2(0,T;X)}=\big(\int_0^T \parallel Tg(t)\parallel_{X}^2 \big)^{\frac{1}{2}}\leq \int_0^T|K(\tau)|\cdot\big( \int_\tau^T \parallel g(t-\tau)\parallel_X^2 dt \big)^{\frac{1}{2}}d\tau\]
    The trick is to apply a substitue: $y:=t-\tau$, then $t=y+\tau$, 
    \[\parallel Tg\parallel_{L^2(0,T;X)}\leq \int_0^T |K(\tau)|d\tau\cdot\parallel g\parallel_{L^2(0,T;X)}\]
    Then the operator norm satisfy:
    \[\parallel T\parallel\leq \int_0^T|K(\tau)|d\tau=\parallel K\parallel_{L^1}\]
    With all these preparations, the next step is to show the solution satisfy the weak form of the equation.
\end{proof}

\begin{thm}
    Let $\{u^N\}\subset L^2(0,T;H_0^1(\Omega))$ be the sequence of the approximation solutions to the problem:
    \[\partial_t u^N=\nabla \cdot (D(x)\int_0^t K(t-s)\nabla \Phi(u^N(x,s)))+f(x,t)\]
    With substitue initial condition: $u^N(0)=u_0^N\to u_0(x)$ strongly in $L^2(\Omega)$ and external force $f\in L^2(0,T;H^{-1}(\Omega))$, then limit function:
    \[u:=\lim u^N\in L^2(0,T;H_0^1(\Omega)\cap C[0,T];L^2(\Omega))\]
    satisfies the weak formulation:
    \[-\int_0^T \int_\Omega u\partial_t\varphi=-\int_0^T\int_\Omega D(x)\bigg(M(u(x,t))\bigg) \cdot\nabla\varphi+\int_0^T\int_\Omega f\varphi\]
    For all test functions $\varphi\in C_c^\infty(\Omega\times (0,T))$.
    \label{Limit function satisfy}
\end{thm}
\begin{proof}
    The proof can be split into 3 parts:
    \begin{enumerate}
        \item \textbf{Force term:} 

        Since $f\in L^2$ and $\varphi$ be fixed, the limit is trivial:
        \[\int_0^T\int_\Omega f\varphi\]
        \item \textbf{Time derivative term:}

        First by \eqref{friedman}:
        \[\int_0^T\int_\Omega u_t^N\varphi=\int_\Omega\int_0^T u_t^N\varphi\]
        By integrate by part:
        \[\int_\Omega\int_0^T u_t^N\varphi=-\int_\Omega \int_0^Tu^N\partial_t \varphi\]
        Again by \eqref{friedman}, there is:
        \[\int_0^T\int_\Omega u_t^N\varphi=-\int_0^T\int_\Omega u^N\partial_t\varphi\to -\int_0^T\int_\Omega u\partial_t\varphi\]
        \item \textbf{Nonlinear Memory term:}

        \eqref{Memory kernel weak convergence} yields weak convergence in $L^2$:
        \[K*\nabla \Phi(u^N)\rightharpoonup K*\nabla \Phi(u)\]
        Since $D(x)$ satisfy uniform ellipticity bounds, in other words: $D\in L^\infty(\Omega)$, hence
        \[\bigg|\int_0^T\int_\Omega D(K*\nabla\Phi(u^N)-K*\nabla\Phi(u))\nabla \varphi\bigg|\leq D_{\max} \bigg|\int_0^T\int_\Omega K*\nabla\Phi(u^N)\nabla \varphi-K*\nabla\Phi(u)\nabla \varphi\bigg|\]
        is a explicit formula for weak convergence.
    \end{enumerate}
    Combining all limits, it can be concluded that the limit $u:=\lim u^N$ satisfies the weak formulation as claimed.
\end{proof}

\section{Discussion}
The global well-posedness of weak solutions for nonlinear memory-type diffusion equations has been demonstrated in this study. The combination of a time-convolution kernel and spatial randomness introduces significant analytical challenges, which we addressed using orthogonal approximation, compactness arguments, and energy estimates.

\ 

\noindent This analysis establishes a foundation for studying nonlocal-in-time diffusion phenomena. Future work may consider extending these results to more general nonlinearities, degenerate or unbounded random coefficients, or stochastic forcing terms. These results may also serve as a theoretical foundation for future developments in homogenization theory and the rigorous analysis of random dynamical systems with memory.

\bibliographystyle{plainnat}
\bibliography{sample}

\end{document}